\documentclass[11pt]{article}

\usepackage[latin1]{inputenc}
\usepackage{t1enc}
\usepackage{amsmath, amssymb, amsfonts, amsthm, amsopn,epsfig}
\usepackage[none]{hyphenat}
\usepackage{float}
\usepackage{graphicx}
\usepackage{caption}
\usepackage[toc,page]{appendix}
\usepackage{epstopdf}
\usepackage{etoolbox}
\usepackage[colorlinks=true]{hyperref}
\usepackage{dsfont}
\hypersetup{
	colorlinks=true,
	linkcolor=blue,
	filecolor=magenta,
	urlcolor=cyan,
	citecolor=blue
}
\usepackage{cleveref}
\usepackage[top=21mm, bottom=22mm, left=21mm, right=21mm]{geometry}
\usepackage{comment}

\theoremstyle{plain}
\newtheorem{theorem}{Theorem}[section]

\newtheorem{corollary}[theorem]{Corollary}
\newtheorem{claim}[theorem]{Claim}
\newtheorem{lemma}[theorem]{Lemma}

\theoremstyle{definition}

\DeclareMathOperator{\surp}{surp}
\DeclareMathOperator{\trace}{tr}

\newcommand{\cF}{\mathcal{F}}

\newcommand{\bR}{\mathbb{R}}

\newcommand{\eps}{\varepsilon}

\newcommand{\hide}[1]{}

\title{Beyond the MaxCut problem in $H$-free graphs}
\author{Zhihan Jin\thanks{Department of Mathematics, ETH Z\"urich, Switzerland. Email: {\tt \{zhihan.jin, aleksa.milojevic\}@math.ethz.ch}. Research supported in part by SNSF grant 200021-228014.}
\and Aleksa Milojevi\'c\footnotemark[1]
\and 
Istv\'an Tomon\thanks{Ume\r{a} University, \emph{e-mail}: \texttt{istvantomon@gmail.com}, Research supported in part by the Swedish Research Council grant VR 2023-03375.}}

\date{}

\begin{document}

\maketitle

\begin{abstract}
In a recent breakthrough, Zhang proves that if $G$ is an $H$-free graph with $m$ edges, then $G$ has a cut of size at least $m/2+c_Hm^{0.5001}$, making a significant step towards a well known conjecture of Alon, Bollob\'as, Krivelevich and Sudakov. We show that the methods of Zhang can be further boosted, and prove the following strengthening. If $G$ is a graph with $m$ edges and no clique of size $m^{1/2-\delta}$, then $G$ has a cut of size at least $m/2+m^{1/2+\eps}$ for some $\eps=\eps(\delta)>0$.

In addition, we sharpen another result of Zhang by proving that if $G$ is an $n$-vertex $m$-edge graph with MaxCut of size at most $m/2+n^{1+\eps}$ (or its smallest eigenvalue $\lambda_n$ satisfies $|\lambda_n|\leq n^{\eps}$), then $G$ is $n^{-\eps}$-close to the disjoint union of cliques for some absolute constant $\eps>0$.
\end{abstract}

\section{Introduction}

Given a graph $G$, a \emph{cut} in $G$ is a partition $(U,V)$ of the vertex set together with all the edges having exactly one endpoint in both parts. The size of the cut is the number of its edges. The \emph{MaxCut} of $G$ is the maximum size of a cut, denoted by $\mbox{mc}(G)$. The MaxCut is among one the most studied 
graph parameters both in extremal combinatorics and theoretical computer science. In this manuscript, we study the extremal and structural properties of graphs with small MaxCut, continuing a line of research initiated in the 1970's.

A simple probabilistic argument shows that every graph with $m$ edges has a cut of size at least $m/2$. Indeed, a random cut, chosen from the uniform distribution on all cuts, has size $m/2$ in expectation. A fundamental result of Edwards \cite{Edwards1,Edwards2} states that this bound can be improved to $\frac{m}{2}+\frac{\sqrt{8m+1}-1}{8}$, which is sharp if the graph is a clique on an odd number of vertices. In general, if $G$ is a disjoint union of constantly many cliques, then the MaxCut of $G$ is of size $m/2+O(\sqrt{m})$. This raises the following natural question. Can this bound be improved if $G$ is ``far'' from a disjoint union cliques? As we are interested in the size of the MaxCut above $m/2$, we focus on the notion of \emph{surplus}, defined as  $\surp(G)=\mbox{mc}(G)-m/2$.

\subsection{MaxCut in H-free graphs}

One way to ensure that a graph is far from a disjoint union of cliques is to assume that it does not contain some fixed graph $H$ as a subgraph. The study of the size of the MaxCut, and in turn the surplus, in such graphs was initiated by Erd\H{o}s and Lov\'asz in the 1970's (see \cite{erdos}). One of the first major results in the area is due to Alon \cite{AlonMaxCut}, who proved that if a graph $G$ has $m$ edges and no triangles, then $\surp(G)=\Omega(m^{4/5})$, and this bound is tight. The size of the surplus in graphs without short cycles is studied in \cite{ABKS,AKS05,BJS,GJS}, where \cite{BJS,GJS} achieve tight bounds for this problem. 
On the other hand, finding the size of the minimum surplus in graphs avoiding $K_r$, the complete graph on $r$ vertices, seems to be much more difficult. Alon, Bollob\'as, Krivelevich, and Sudakov \cite{AKS05} proved that for every $r$, there exists $\eps_r>0$ such that every $K_r$-free graph has surplus at least $m^{1/2+\eps_r}$. This was improved by Carlson, Kolla, Li, Mani, Sudakov, and Trevisan \cite{CKLMST}, and then by Glock, Janzer, and Sudakov \cite{GJS} proved the bound $m^{1/2+\frac{3}{4r-2}}$. However, it is conjectured in \cite{ABKS} that the answer should be $m^{3/4+\eps_r}$ for some $\eps_r>0$. This conjecture is still wide open, and for a long time, it was a tantalizing open problem to find any absolute constant $\eps>0$ (independent of $r$), such that every $K_r$-free graph has surplus $\Omega_r(m^{1/2+\eps})$. In a recent breakthrough, Zhang \cite{Zhang} proved exactly this, by showing that one can take $\eps=0.001$.

Forbidding a fixed graph $H$ as a subgraph is a very strong, local restriction on the graph. It seemed plausible that much weaker global restrictions might also force large surplus. One such result is due to Glock, Janzer, and Sudakov \cite{GJS} who show that the surplus depends sensitively on the number of triangles. Another global result, due to R\"aty, Sudakov, and Tomon \cite{RST} shows that regular graphs of density between $1/2+\delta$ and $1-\delta$ have surplus $\Omega(m^{5/8})$. In the very extreme case, Balla, Hambardzumyan, and Tomon \cite{BHT} showed that graphs with clique number $o(\sqrt{m})$ already have surplus $\omega(m^{1/2})$. Here, we greatly strengthen the result of Zhang \cite{Zhang} and the aforementioned result of \cite{BHT} by proving that clique number $m^{1/2-\delta}$ already implies that the surplus is at least $m^{1/2+\eps}$ for some $\eps=\eps(\delta)>0$.

\begin{theorem}\label{thm:maxcut}
For very $\eps>0$ there exists $\delta>0$ such that the following holds for every sufficiently large $m$. Let $G$ be a graph with $m$ edges such that $G$ contains no clique of size $m^{1/2-\delta}$. Then $G$ has a cut of size at least $m/2+m^{1/2+\eps}$. 
\end{theorem}

We briefly discuss the methods used in the proof of this theorem. While early results on the surplus used mostly combinatorial and probabilistic techniques, the celebrated approximation algorithm of Goemans and Williamson \cite{GW95} introduced a semidefinite programming inspired approach. In the context of extremal properties of the MaxCut, this approach was first used by Carlson, Kolla, Li, Mani, Sudakov, and Trevisan \cite{CKLMST}, and then further refined by Glock, Janzer, and Sudakov \cite{GJS} and Balla, Janzer, and Sudakov \cite{BJS}. Later, R\"aty, Sudakov, and Tomon \cite{RST} demonstrated that one can combine this approach with spectral techniques to achieve results that are seemingly unobtainable with other methods. Finally, Zhang's \cite{Zhang} paper is a tour de force of spectral analysis of graphs. In order to prove Theorem \ref{thm:maxcut}, we build on certain results of Zhang as a blackbox, and apply spectral methods similar to the ones developed in \cite{RST}. We highlight that the result of Balla, Hambardzumyan, and Tomon \cite{BHT} finding surplus $\omega(m^{1/2})$ in graphs with clique number $o(m^{1/2})$ uses techniques that are fundamentally different from these approaches, it builds on properties of factorization norms. Also, there is no implication between this result and Theorem \ref{thm:maxcut}, despite the similarity.

\subsection{Structural properties of graphs with small cuts}

Another, more direct way to measure whether a graph is far from a disjoint union of cliques is to consider the edit-distance from a union of cliques. We say that an $n$-vertex graph $G$ is \emph{$\mu$-close} to some family of graphs $\cF$ if it is possible to add or remove at most $\mu n^2$ edges of $G$ to get a member of $\mathcal{F}$. In other words, the \emph{edit-distance} of $G$ to $\cF$ is at most $\mu n^2$.

We first heard about the following problem from Victor Falgas-Ravry (personal communication). Is it true that if an $n$-vertex graph $G$ is $\mu$-far from a disjoint union of cliques, then $G$ has surplus at least $\Omega_{\mu}(n^{3/2})$? This conjecture turns out to be false, there are several families of strongly-regular and other algebraically defined graph families achieving smaller surplus. For example, a construction of de Caen \cite{caen} related to equiangular lines achieves surplus $O(n^{5/4})$, while it is constant far from the disjoint union of cliques. A closely related conjecture of \cite{RST} is that if $G$ is $\mu$-far from a Tur\'an graph, then $G$ must have \emph{positive discrepancy} at least $\Omega(n^{5/4})$. Here, the positive discrepancy can be thought of as the surplus of the complement graph (they are equal up to constant factors if $G$ is regular, but otherwise there are some subtle differences), and the complement of a Tur\'an graph is the disjoint union of cliques of equal sizes. Zhang \cite{Zhang} made significant progress towards this conjecture as well, and also proved that every $n$-vertex graph of surplus at most $O_{\mu}(n^{1.001})$ is $\mu$-close to a disjoint union of cliques. Zhang asked whether the dependence of $O_{\mu}(n^{1.001})$ on $\mu$ can be chosen to be polynomial. We answer this question in the affirmative.

\begin{theorem}\label{thm:stability}
There exists $\eps>0$ such that the following holds for every sufficiently large $n$. If $G$ is an $n$-vertex $m$-edge graph with no cut of size larger than $m/2+n^{1+\eps}$, then $G$ is $n^{-\eps}$-close to the disjoint union of cliques.
\end{theorem}

\subsection{Smallest eigenvalue}

A central topic of spectral graph theory is understanding the structure of graphs, whose adjacency matrix has large smallest eigenvalue, see e.g. Koolen, Cao, and Yang \cite{KCY} for a recent survey. Let $G$ be an $n$-vertex graph and let $\lambda_n$ denote the smallest eigenvalue of its adjacency matrix. If $G$ is non-empty, then $\lambda_n\leq -1$ with equality if and only if $G$ is the disjoint union of cliques.  In the 1970's,  Cameron, Goethels, Seidel, and Shult \cite{CGSS}
gave a complete characterization of graphs satisfying  $\lambda_n\geq -2$,  while more recently,
Koolen, Yang and Yang \cite{KYY} obtained a partial characterization in the case $\lambda_n\geq -3$. Beyond these specific values, much less is known. Kim, Koolen, and Yang \cite{KKY} proved the following structure theorem for graphs satisfying $|\lambda_n|\leq \lambda$. One can find dense induced subgraph $Q_1,\dots,Q_c$ in $G$ such that each vertex lies in at most $\lambda$ of $Q_1,\dots,Q_c$, and almost all edges are covered by the union of $Q_1,\dots,Q_c$. However, the proof of this result is based on certain Ramsey theoretic arguments, and it no longer applies if $\lambda$ is polynomial in~$n$.

On the other hand, Zhang \cite{Zhang} proved that as long as $|\lambda_n|<n^{1/4-\eps}$ for some $\eps>0$, then the graph $G$ is $\mu$-close to the disjoint union of cliques for every $\mu>0$, assuming $n>n_0(\eps,\mu)$ is sufficiently large. Also, the exponent $1/4$ cannot be replaced by anything larger, since de Caen \cite{caen} constructs graphs which are $\mu$-far from a disjoint union of cliques for $\mu=\Omega(1)$ and have $|\lambda_n|=\Theta(n^{1/4})$. This result essentially resolved another conjecture of \cite{RST}, and it gives a strong structural property of dense graphs with small $|\lambda_n|$. On the other hand, this result is meaningless for sparse graphs, since any graph with fewer than $\mu n^2$ edges is $\mu$-close to a disjoint union of cliques. 

We observe that the following corollary of Theorem~\ref{thm:stability} extends the result of Zhang \cite{Zhang} to a sparser regime. 

\begin{corollary}
There exists $\eps>0$ such that the following holds for every sufficiently large $n$. If $G$ is an $n$-vertex graph with whose smallest eigenvalue $\lambda_n$ satisfies $|\lambda_n|\leq n^{\eps}$, then $G$ is $n^{-\eps}$-close to the disjoint union of cliques.
\end{corollary}

The surplus and the smallest eigenvalue are closely related. It is well known (see e.g. \cite{Zhang} or Claim \ref{claim:surplus_and_lambdan}) that $\surp(G)\leq n|\lambda_n|$. Thus, Theorem~\ref{thm:stability} immediately applies to any graph with $|\lambda_n|\leq n^{\eps}$ since such graphs also have $\surp(G)\leq n^{1+\eps}$, thus showing the above corollary.

We also note that one can think of the surplus as a robust version of the smallest eigenvalue. Indeed, $\surp(G)\approx n|\lambda_n|$ holds in many natural graph families, but the surplus is much less sensitive to local changes than $\lambda_n$.
 % In light of the bound $\surp(G)\leq n|\lambda_n|$, Theorem \ref{thm:stability} immediately implies the next statement, which extends the result of Zhang \cite{Zhang} to a sparser regime. 

We prove Theorem \ref{thm:maxcut} in Sections \ref{sect:1} and Section \ref{sect:2}, and then we prove Theorem \ref{thm:stability} in Section \ref{sect:stability}. But before that, we introduce our notation and discuss some basic results in the next section.

\section{Preliminaries}

In this paper, we crucially exploit the relation between the MaxCut of a graph and its spectral properties. Therefore, we will start by recalling  some basic facts from linear algebra and graph theory, as well as some known relationships between spectra of graphs and their MaxCut. 

The \emph{edge density} of a graph $G$ is $m/\binom{n}{2}$, where $m=e(G)$ is the number of edges. Given a subset $U$ of the vertices, $G[U]$ denotes the subgraph of $G$ induced on vertex set $U$. Also, if $V\subset V(G)$ is disjoint form $U$, then $G[U,V]$ is the bipartite subgraph of $V(G)$ induced between $U$ and $V$. The \emph{complement} of $G$ is denoted by $\overline{G}$. The \emph{maximum degree} of $G$ is denoted by $\Delta(G)$, and the average degree by $d(G)$.

The \emph{MaxCut} of $G$, denoted by $\mbox{mc}(G)$, is the maximum size of a cut, where a \emph{cut} is a partition $(U,V)$ of the vertices into two parts, with all the edges having exactly one endpoint in both parts. The size of a cut is the number of its edges. The \emph{surplus} of $G$ is defined as $\surp(G)=\mbox{mc}(G)-m/2$, where $m$ is the number of edges of $G$. Note that $\surp(G)$ is always nonnegative. A useful property of the surplus is that if $G_0\subset G$, then $\surp(G)\geq \surp(G_0)$, see e.g. \cite{GJS}.

As we will see, the surplus of $G$ is controlled by the negative eigenvalues of the graph. If $G$ is an $n$-vertex graph whose adjacency matrix is $A$, then we denote by $\lambda_1\geq\dots\geq \lambda_n$ the eigenvalues of $A$, sometimes also calling them the eigenvalues of $G$. We also denote by $v_1, \dots, v_n$ the corresponding orthonormal basis of eigenvectors. By the Perron-Frobenius theorem, we may take $v_1$ to be a vector with non-negative entries, which we call the \emph{principal eigenvector} of $A$. Furthermore, the corresponding eigenvalue satisfies $\lambda_1\geq d(G)$. 

Finally, before presenting the relations among the eigenvalues and the surplus, we introduce a few final pieces of notation about matrices. Given two $n\times n$ matrices $A$ and $B$, their scalar product is defined as $\langle A,B\rangle=\trace(AB^T)=\sum_{1\leq i,j\leq n} A_{i,j}B_{i,j}.$
The \emph{Frobenius-norm} of an $n\times n$ matrix $A$ is
$$\|A\|_F^2=\langle A,A\rangle=\sum_{i,j=1}^{n}A_{i,j}^2.$$
If $A$ is symmetric with  eigenvalues  $\lambda_1,\dots,\lambda_n$, then we also have $\|A\|_F^2= \langle A,A\rangle = \trace(A^2)=\sum_{i=1}^n\lambda_i^2$.

The \emph{Hadamard product} (also known as entry-wise product) of $A$ and $B$ is the $n\times n$ matrix $A\circ B$ defined as $(A\circ B)_{i,j}=A_{i,j}B_{i,j}$. We denote the $k$-term Hadamard product $A\circ\dots\circ A$ by $A^{\circ k}$. A useful feature of the Hadamard product, which is a key component of our arguments, is that Hadamard product preserves positive semidefiniteness. More precisely, the \emph{Schur product theorem} states that if $A$ and $B$ are positive semidefinite symmetric matrices, then $A\circ B$ is also positive semidefinite. We remark that the Hadamard product played a crucial role in both \cite{RST} and \cite{Zhang}. In particar, Zhang \cite{Zhang} builds on the simple observation that if $A$ is an adjacency matrix, then $A=A\circ A$. In the current manuscript, one of the key steps is analyzing certain three-term Hadamard products.

We now begin by showing the promised relation between the MaxCut of a graph and its spectrum. We begin with an upper bound.

\begin{claim}\label{claim:surplus_and_lambdan}
For an $n$-vertex graph $G$ with the smallest eigenvalue $\lambda_n$, we have $\surp(G)\leq |\lambda_n|n/4$. 
\end{claim}
\begin{proof}
Let $A$ be the adjacency matrix of $G$. We can assign a vector with entries $\pm 1$ to each cut $V(G)=X\cup Y$, by setting $x_u=1$ if $u\in X$ and $x_u=-1$ otherwise. Then, the surplus of this cut equals $\frac{1}{2}\big(e(X, Y)-e(X)-e(Y)\big)=-\frac{1}{2}\sum_{\{u, v\}\subseteq V(G)} x_uA_{uv}x_v=-\frac{1}{4}\sum_{u, v\in V(G)} x_uA_{uv}x_v$. 
Hence, we have 
\[\surp(G)=\frac{1}{4}\max_{x\in \{-1,1\}^n}-x^T Ax=\frac{1}{4}\max_{x\in [-1,1]^n}-x^T Ax.\]
But, we have $-x^TAx\leq |\lambda_n|\|x\|_2^2$ for every vector $x\in \bR^n$, and so $\surp(G)\leq \frac{1}{4} |\lambda_n| \sqrt{n}^2=|\lambda_n|n/4$.
\end{proof}

The key ingredient of the above proof is the relation $\surp(G)=\frac{1}{4}\max_{x\in [-1,1]^n}-x^T Ax$, which can also be written as $\surp(G)=\frac{1}{4}\max_{x\in [-1, 1]^n}\langle -A, xx^T\rangle$, where we observe that $xx^T$ is a positive-semidefinite matrix with diagonal entries bounded by $1$. As we will see, it will be very convenient to define the semidefinite relaxation of the surplus as follows. Given an $n$-vertex graph $G$ with adjacency matrix $A$, define $$\surp^*(G)=\max_{X} -\langle A,X\rangle,$$ where the maximum is taken over all $n\times n$ positive semidefinite matrices $X$ such that $X_{i,i}\leq 1$ for every $i\in [n]$. The following inequality between $\surp(G)$ and $\surp^*(G)$ can be found in \cite{RT} and \cite{Zhang}, and it is a simple application of the graph Grothendieck inequality of Charikar and Wirth \cite{CW04}. 

\begin{claim}[\cite{RT}]
For every graph $G$, we have $\surp^*(G)\geq \surp(G)\geq \Omega\Big(\frac{\surp^*(G)}{\log n}\Big)$.
\end{claim}

The semidefinite relaxation $\surp^*(G)$ allows us to obtain lower-bounds on the surplus using the negative eigenvalues of a graph $G$. In this paper, we will be often concerned with very dense graphs, and the following lemma gives good lower bounds on the surplus of $G$ in this case. Parts of the following lemma and similar bounds can be also found in \cite{RST,RT,Zhang}.

\begin{lemma}\label{lemma:surp_star}
Let $G$ be a graph on $n$ vertices with eigenvalues $\lambda_1\geq \dots\geq \lambda_n$, and let $\Delta$ denote the maximum degree of the complement of $G$. Then
\begin{itemize}
    \item[(i)] $\surp^*(G)\geq \sum_{\lambda_i<0}|\lambda_i|$
    \item[(ii)] $\surp^*(G)\geq \Omega\left(\frac{1}{\sqrt{\Delta+1}}\sum_{\lambda_i<0}\lambda_i^2\right)$
    \item[(iii)]  $\surp^*(G)\geq \Omega\left(\frac{1}{\Delta+1}\sum_{\lambda_i<0}|\lambda_i|^3\right)$
\end{itemize}
\end{lemma}

Before we prove Lemma~\ref{lemma:surp_star}, we briefly discuss a preliminary lemma. We show that in dense graphs, the principal eigenvector must be close to the normalized all-ones vector.

\begin{lemma}\label{lemma:max_entry}
Let $G$ be an $n$-vertex graph, whose complement has edge density $p\leq 1/10$ and maximum degree $\Delta=\Delta(\overline{G})$. If $v_1$ is the principal eigenvector of $G$, then for each $i\in [n]$ we have \[\frac{1-2\Delta/n}{\sqrt{n}} \leq v_1(i)\leq \frac{1+2p+2/n}{\sqrt{n}}.\]
\end{lemma}
\begin{proof}
Let $\lambda_1$ be the eigenvalue corresponding to $v_1$. Then for every $b\in [n]$, we have $\lambda_1v_1(b)=\sum_{b\sim i}v_1(i)$. By the inequality between the arithmetic and quadratic mean, we have $\frac{1}{n}\sum_{b\sim i}v_1(i)\leq \frac{1}{n}\sum_{i=1}^n v_1(i)\leq \sqrt{\frac{\sum_i v_1(i)^2}{n}}=1/\sqrt{n}$, where we have used that $\sum_{i=1}^nv_1(i)^2=1$. Hence $\lambda_1v_1(b)\leq \sqrt{n}$ and so
$$v_1(b)\leq \frac{\sqrt{n}}{\lambda_1}\leq \frac{\sqrt{n}}{d}=\frac{\sqrt{n}}{(1-p)(n-1)}\leq \frac{1+2p+2/n}{\sqrt{n}}.$$
Here, in the second inequality, we used that $\lambda_1\geq d(G)$, and in the last inequality that $p<1/10$. To prove the lower bound, we first observe that 
$$1=\sum_{i=1}^n v_1(i)^2\leq \|v_1\|_{\infty} \sum_{i=1}^n v_1(i),$$
which implies that $\sum_{i=1}^nv_1(i)\geq \frac{\lambda_1}{\sqrt{n}}$. But then using the identity $Av_1=\lambda_1 v_1$,
$$\lambda_1v_1(b)=\sum_{i\sim b}v_1(i)\geq \sum_{i=1}^nv_1(i)-\Delta \|v_1\|_{\infty}\geq\frac{\lambda_1}{\sqrt{n}}-\Delta\frac{\sqrt{n}}{\lambda_1}\geq \frac{\lambda_1}{\sqrt{n}}\Big(1-\Delta\frac{n}{\lambda_1^2}\Big)\geq \frac{\lambda_1}{\sqrt{n}}\Big(1-\frac{2\Delta}{n}\Big),$$
where we have used that $\lambda_1^2\geq d(G)^2\geq n^2/2$ in the last inequality. Canceling $\lambda_1$ gives $v_1(b)\geq \frac{1-2\Delta/n}{\sqrt{n}}$.
\end{proof}

\begin{proof}[Proof of Lemma~\ref{lemma:surp_star}.]
We begin by showing the inequalities \textit{(i)} and \textit{(iii)}, which we then combine to derive \textit{(ii)}. Let $v_1,\dots,v_n$ be an orthonormal basis of eigenvectors corresponding to the eigenvalues $\lambda_1,\dots,\lambda_n$. The inequalities \textit{(i)} and \textit{(iii)} will be shown by plugging in the appropriate test matrix $X$ in the formula $\surp^*(G)= \max_X -\langle A, X\rangle$. Observe that, if we choose $X=\sum_{i=1}^n\alpha_i v_iv_i^T$ for some real numbers $\alpha_1,\dots,\alpha_n$, then
$$\langle A,X\rangle=\sum_{i=1}^n\sum_{j=1}^n\alpha_i\lambda_j \langle v_iv_i^T,v_jv_j^T\rangle=\sum_{i=1}^n\sum_{j=1}^n\alpha_i\lambda_j \langle v_i,v_j\rangle^2=\sum_{i=1}^n\alpha_i\lambda_i.$$

\textit{(i)} Let $X=\sum_{\lambda_i<0} v_iv_i^T$. Then $X$ is positive semidefinite, and as $v_1,\dots,v_n$ is an orthonormal basis, we have $$X_{j,j}=\sum_{\lambda_i<0}v_i(j)^2\leq\sum_{i=1}^n v_i(j)^2=\sum_{i=1}^n \langle v_i, \mathbf{e}_j\rangle^2= \|\mathbf{e}_j\|^2=1.$$ Therefore, \[\surp^*(G)\geq -\langle A,X\rangle=\sum_{\lambda_i<0}|\lambda_{i}|.\]

\textit{(iii)} Let $\beta=\frac{1}{100(\Delta+1)}$, and $X=\beta\sum_{\lambda_i<0} \lambda_i^2v_iv_i^T$. Then $X$ is positive semidefinite. It is enough to prove that the diagonal entries of $X$ are bounded by $1$, as then $\surp^*(G)\geq -\langle A,X\rangle =\beta\sum_{\lambda_i<0}|\lambda_{i}|^3$. 

To show that  $X_{i,i}\leq 1$, we consider the matrix $B=A-\lambda_1v_1v_1^T$. Since $$\beta B^2-X=\beta\sum_{i\neq 0,\lambda_i>0}\lambda_i^2v_iv_i^T,$$ we have that $\beta B^2-X$ is positive semidefinite. Hence, $(\beta B^2)_{i,i}\geq X_{i,i}$ for every $i\in [n]$. Therefore, it is enough to show that $(B^2)_{i,i}\leq 1/\beta=100(\Delta+1)$. 

To show this, we first bound the entries of $B$. We denote by $p$ the density of $\overline{G}$, and observe that $p\leq \frac{\Delta n/2}{\binom{n}{2}}\leq \frac{\Delta}{n-1}$. Then, Lemma \ref{lemma:max_entry} implies that for any $i, j\in [n]$ we have 
\[ 1-\frac{5\Delta}{n}\leq (n-1)(1-p)\left(\frac{1-2\Delta/n}{\sqrt{n}}\right)^2\leq \lambda_1 v_1(i)v_1(j)\leq n\left(\frac{1+2p+2/n}{\sqrt{n}}\right)^2\leq 1+\frac{5(\Delta+1)}{n}.\]
Therefore, for every $i, j\in [n]$, if $ij\in E(G)$ and $A_{i, j}=1$, then $|B_{i, j}|\leq \frac{5(\Delta+1)}{n}$. Otherwise, we have $|B_{i,j}|\leq 1+5(\Delta+1)/n\leq 6$. From this, we have 
$$(B^2)_{i,i}=\sum_{j=1}^n (B_{i,j})^2\leq 36\Delta+n\frac{25(\Delta+1)^2}{n^2}\leq 100(\Delta+1).$$
\textit{(ii)} We show that \textit{(i)} and \textit{(iii)} can be combined to give the desired lower bound on $\surp^*(G)$. Namely, we have
\[\surp^*(G)^2\geq \beta\left(\sum_{\lambda_i<0}|\lambda_i|^3\right)\left(\sum_{\lambda_i<0}|\lambda_i|\right)\geq \beta\left(\sum_{\lambda_i<0}\lambda_i^2\right)^2.\]
Note that the first inequality is the combination of \textit{(i)} and \textit{(iii)}, while the second one is simply the Cauchy-Schwartz inequality applied to the sequences $(|\lambda_i|^3)_{\lambda_i<0}$ and $(|\lambda_i|)_{\lambda_i<0}$. Taking square roots then proves \textit{(ii)}.
\end{proof}

Finally, we remark that $\surp^*(\cdot)$ is also monotone, that is, if $G_0$ is an induced subgraph of $G$, then $\surp^*(G_0)\leq \surp^*(G)$. This fact follows easily from the definition, and it is used repeatedly throughout our paper.

\section{Dense graphs with small surplus}\label{sect:1}

The proof of Theorem \ref{thm:maxcut} can be summarized in the following simple steps, in each of which we find increasingly denser subgraphs of our host graph. Interestingly, the underlying methods dealing with each of these steps are quite different. Let $G$ be a graph on $n$ vertices with $m$ edges and assume that $\surp(G)\leq m^{1/2+\eps}$. Without loss of generality, $G$ has no isolated vertices.

First, by an old result of Erd\H{o}s, Gy\'arf\'as, and Kohayakawa \cite{EGyK}, $\surp(G)=\Omega(n)$. Therefore, we may assume that $m>n^{2-\rho}$ for some $0<\rho=\rho(\eps)$. Second, by a result of Zhang \cite{Zhang}, such a dense graph $G$ contains a subgraph $G_1$ on at least $n^{2-\rho}$ vertices of edge density $1-10^{-5}$. Interestingly, in \cite{Zhang}, this step consists of two sub-steps: getting a subgraph $G_{0.5}$ of density $\Omega(1)$, and then getting a subgraph $G_1$ of $G_{0.5}$ using stability. Third, we prove that $G_1$ contains an induced subgraph $G_2$ of edge density at least $1-n^{-0.1}$ on $v(G_1)^{1-o(1)}$ vertices. This result is presented as Lemma \ref{lemma:main}. Finally, we show that $G_2$ contains a half-sized induced subgraph $G_3$ of edge density at least $1-n^{O(\eps)-1}$. This result is presented as Lemma \ref{lemma:very_dense}.

As outlined above, one of the key components of our proof is the following remarkable result of Zhang \cite{Zhang}, which we use as a black box.

\begin{theorem}[Theorem 1.7 in \cite{Zhang}]\label{thm:zhang}
Let $\rho<\frac{1}{1000}$. For every $p>0$, the following holds for every sufficiently large $n$ with respect to $p$. Let $G$ be an $n$-vertex graph with at least $n^{2-\rho}$ edges. If $\surp(G)\leq n^{1+\rho}$, then $G$ contains an induced subgraph on at least $n^{1-4\rho}$ vertices of edge density at least $(1-p)$.
\end{theorem}

Hence, in order to prove Theorem \ref{thm:maxcut}, we may pass to a very dense subgraph. Unfortunately, the dependence of $n$ on $p$ for which Theorem \ref{thm:zhang} starts to be effective is exponential. However, we can boost this result even further. The main result of this section shows that if $G$ has density at least $1-10^{-5}$, then either $G$ has large surplus, or a large induced subgraph of density at least $1-n^{-0.1}$. Then, in the next subsection, we boost this even further to show that graphs of density $1-n^{-0.1}$ either have large surplus, or they contain a large subgraph of density at least $1-n^{-0.99}$.

The main result of this section is the following lemma.

\begin{lemma}\label{lemma:main}
Let $0<\eps<1/4$ and $0<\alpha<\min\{\frac{1}{12}-\frac{\eps}{6},\frac{1}{6}-\frac{2\eps}{3}\}$, and let $n$ be sufficiently large with respect to $\eps,\alpha$. Let $G$ be an $n$-vertex graph with edge density at least $1-10^{-5}$. If $\surp^*(G)<n^{1+\eps}$, then $G$ contains an induced subgraph of size $n^{1-o(1)}$ with edge density at least $1-n^{-\alpha}$.
\end{lemma}

% We prepare the proof of this theorem with a series of lemmas. First, we show that if $G$ is a very dense graph, then the entries of the principle eigenvector cannot be much larger than $1/\sqrt{n}$.

In order to prove Lemma \ref{lemma:main}, we proceed by a density increment argument. The next lemma contains the statement of the main density increment step.

\begin{lemma}\label{lemma:density_increment}
Let $0<\eps<1/4$ and $0<\alpha<\min\{\frac{1}{12}-\frac{\eps}{6},\frac{1}{6}-\frac{2\eps}{3}\}$, and let $n$ be sufficiently large with respect to $\eps,\alpha$. Let $G$ be an $n$-vertex graph with edge density $1-p$, where $n^{-\alpha}\leq p<10^{-5}$. Assume that $\surp^*(G)<n^{1+\eps}$. Then $G$ contains an $n/4$-vertex induced subgraph of edge density at least $1-10^8 p^3$.
\end{lemma}

\begin{proof}
Let $A$ be the adjacency matrix of $G$ with eigenvalues $\lambda_1\geq \dots\geq \lambda_n$ and corresponding orthonormal basis of eigenvectors $v_1,\dots,v_n$. Define $B=A-\lambda_1v_1v_1^T$ and $E=\sum_{\lambda_i<0}|\lambda_i| v_iv_i^T$. Then, the matrices $E$ and $B+E=\sum_{\lambda_i>0,i\neq 1}\lambda_i v_iv_i^T$ are positive semidefinite. 

The key idea of the proof is to consider the following triple Hadamard product:
$$D=(B+E)^{\circ 3}=B^{\circ 3}+3B\circ B\circ E+3B\circ E\circ E+E^{\circ 3}.$$
As $B+E$ is positive semidefinite, so is $D$ by the Schur product theorem. Then, using the matrix $E$ and the principal eigenvector $v_1$, we identify a set of well-behaved vertices $I$, and evaluate the product \[0\leq \mathds{1}_I^T D \mathds{1}_I= \mathds{1}_I^TB^{\circ 3}\mathds{1}_I+\mathds{1}_I^T\big(3B\circ B\circ E+3B\circ E\circ E+E^{\circ 3}\big) \mathds{1}_I.\]
By carefully analyzing the terms of this product, we will conclude that the graph $\overline{G}[I]$ must be \textit{much} sparser than $G$.

We will now give the details. First, observe that 
$$\trace(E)=\sum_{\lambda_i<0}|\lambda_i|\leq \surp^*(G)\leq n^{1+\eps}$$
by \textit{(i)} in Lemma \ref{lemma:surp_star}, and 
$$\|E\|_F^2=\sum_{\lambda_i<0}\lambda_i^2\leq O(\sqrt{\Delta+1})\surp^*(G)\leq O(\sqrt{\Delta+1}) n^{1+\eps}\leq O(n^{3/2+\eps})$$
by \textit{(ii)} in Lemma \ref{lemma:surp_star}. Let $I$ be the set of vertices $i\in [n]$ that satisfy $E_{i,i}\leq 4n^{\eps}$ and $v_1(i)\geq (1-8p)/\sqrt{n}$.
\begin{claim}
$|I|\geq n/4$.
\end{claim}
\begin{proof}
Let $I_0$ be the set of vertices $i\in [n]$ such that $v_1(i)\geq (1-8p)/\sqrt{n}.$ Then using Lemma \ref{lemma:max_entry},
\begin{align*}
    1&=\sum_{i=1}^nv_1(i)^2\leq |I_0|\frac{(1+2p+2/n)^2}{n}+(n-|I_0|)\frac{(1-8p)^2}{n}\\
    &\leq \frac{1}{n}\left(|I_0|(1+8p)+(n-|I_0|)(1-8p)\right)=(1-8p)+\frac{16p|I_0|}{n}.
\end{align*}
From this, we get $|I_0|\geq n/2$. Now $I$ is those set of vertices $i\in I_0$ that satisfy $E_{i,i}\leq 4n^{\eps}$. As $\trace(E_{i,i})\leq n^{1+\eps}$, the number of vertices such that $E_{i,i}>4n^{\eps}$ is at most $n/4$, giving the desired bound $|I|\geq n/4$. 
\end{proof}

We now evaluate the terms of the $\mathds{1}_I^T D \mathds{1}_I$, starting with the main term $\mathds{1}_I^T B^{\circ 3} \mathds{1}_I$. 
\begin{claim} \label{claim:main term}
$\mathds{1}^T_I B^{\circ 3}\mathds{1}_I\leq 10^6|I|^2p^3-e(\overline{G}[I])/4$.
\end{claim}
\begin{proof}
We have 
$$B_{i,j}=\begin{cases}
    1-\lambda_1v_1(i)v_1(j) &\mbox{if }ij\in E(G),\\
    -\lambda_1v_1(i)v_1(j) &\mbox{if }ij\notin E(G).
\end{cases}$$
If $i,j\in I$, then
$$1-\lambda_1v_1(i)v_1(j)\leq 1-(1-p)(n-1)\cdot \left(\frac{1-8p}{\sqrt{n}}\right)^2\leq 100p,$$
and  thus $\lambda_1v_1(i)v_1(j)>\frac{1}{2}$. Therefore,
$$\mathds{1}^T_I B^{\circ 3}\mathds{1}_I=\sum_{i,j\in I, i\sim j} (1-\lambda_1v_1(i)v_1(j))^3-\sum_{i,j\in I, i\not\sim j}(\lambda_1v_1(i)v_1(j))^3\leq 100^3|I|^2p^3-2e(\overline{G}[I])/8.$$
\end{proof}

\begin{claim} \label{claim:error term}
We have
\[\mathds{1}^T_I (3B\circ B\circ E+3B\circ E\circ E+E^{\circ 3})\mathds{1}_I\leq O(n^{7/4+\eps/2}+n^{3/2+2\eps}).\]
\end{claim}
\begin{proof}
We bound each summand of the error term independently. Firstly, note that every entry of $B$ is between $1$ and $-2$, as $0\leq \lambda_1v_1(i)v_2(j)\leq n\left(\frac{1+2p+2/n}{\sqrt{n}}\right)^2\leq 2$. Therefore, we have 
$$\mathds{1}^T_I (3B\circ B\circ E)\mathds{1}_I\leq 12\sum_{i,j\in I} |E_{i,j}|\leq 12|I|\sqrt{\sum_{i,j\in I}E_{i,j}^2}\leq 12n\|E\|_F\leq O(n^{7/4+\eps/2}).$$
Here, the second inequality holds by the inequality between the arithmetic and quadratic mean.

To bound the second summand, we again use that entries of $B$ are bounded by $2$ in absolute value, and so
$$\mathds{1}^T_I (3B\circ E\circ E)\mathds{1}_I\leq 6\sum_{i,j\in I} E_{i,j}^2\leq 6\|E\|_F^2\leq O(n^{3/2+\eps}).$$
Finally, in bounding the last term we will use that $E_{i, i}\leq n^\eps$ for all $i\in I$. In particular, since $E$ is a positive definite matrix, this implies that $|E_{i, j}|\leq n^{\eps}$ for all $i, j\in I$. So,
$$\mathds{1}^T_I E^{\circ 3}\mathds{1}_I\leq \sum_{i,j\in I} |E_{i,j}|^3\leq \max_{i,j\in I} |E_{i,j}|\cdot \sum_{i,j\in I}|E_{i,j}|^2\leq n^\eps \cdot \|E\|_F^2\leq O(n^{3/2+2\eps}).$$
The conclusion now follows by summing up the bounds obtained on each of the error terms above.
%Using that $E$ is positive semidefinite and that $E_{i,i}\leq 4n^{\eps}$ for every $i\in I$, we have $\max_{i,j\in I} |E_{i,j}|\leq 4n^{\eps}$. Therefore, the claim follows by our upper bound on $\|E\|_F^2$.
\end{proof}

In conclusion, we proved that 
$$0\leq \mathds{1}^T_I D\mathds{1}_I\leq 10^6|I|^2p^3-e(\overline{G}[I])/4+O(n^{7/4+\eps/2}+n^{3/2+2\eps})\leq 10^7|I|^2p^3-e(\overline{G}[I])/4,$$
where in the last inequality we used our bounds on $p$ and $\eps$, and that $n$ is sufficiently large. Hence, we must have $e(\overline{G}[I])\leq 10^8 |I|^2p^3/2$, which shows that the edge density of $G[I]$ is at least $(1-10^8p^3)$.
\end{proof}

\begin{proof}[Proof of Lemma \ref{lemma:main}]
    Let $\eps_0>\eps$ and $\alpha_0>\alpha$ be real numbers (independent of $n$) such that $\eps_0<1/4$ and $\alpha_0<\min\{\frac{1}{12}-\frac{\eps_0}{6},\frac{1}{6}-\frac{2\eps_0}{3}\}$. Let $1-p$ be the edge density of $G$. Let $G_0=G$, and define the sequence of induced subgraphs $G_0\supset G_1\supset...$ as follows. If $G_i$ is already defined with $n_i$ vertices and edge density $1-p_i$, then stop if either $p_i<n^{-\alpha}$ or $n_i<\max\{n^{(1+\eps)/(1+\eps_0)},n^{\alpha/\alpha_0}\}$. Otherwise, let $G_{i+1}$ be an induced subgraph of $G_i$ on at least $n_i/4$ vertices of edge density at least $1-10^8p_i^3$. 

    First, we argue that if we did not stop at $G_i$, we can indeed find a suitable subgraph $G_{i+1}$. Note that $\surp^*(G)\geq \surp^*(G_i)$, so if $n_i\geq \max\{n^{(1+\eps)/(1+\eps_0)},n^{\alpha/\alpha_0}\}$, then $\surp^*(G_i)\leq n^{1+\eps}<n_i^{1+\eps_0}$. Moreover, using that $10^8p^3<p$ for $p<10^{-5}$, it follows by induction that the edge density of $G_i$ is at least $1-10^{-5}$. Hence, if $p_i>n^{-\alpha}\geq n_i^{-\alpha_0}$, then  an application of Lemma \ref{lemma:density_increment} with $\eps_0$ and $\alpha_0$ instead of $\eps$ and $\alpha$ guarantees the existence of $G_{i+1}$.
    
     Let $I$ be the last index $i$ for which $G_i$ is defined. Then $n_I\geq n4^{-I}$ and $p_{I}<(10^{8})^{3^{I-1}+3^{I-2}+\dots+1}p^{3^I}<(10^4p)^{3^I}<10^{-3^I}$. Hence, we must have stopped because $p_{I}<n^{-\alpha}$, which happens for some $I=O(\log \log n)$. But then $n_I\geq n4^{-O(\log\log n)}=n^{1-o(1)}$, and thus $G_I$ suffices.
\end{proof}

\section{Even denser graphs with small surplus}\label{sect:2}

The main result of this section is the following lemma, which can be used to further boost Lemma \ref{lemma:main}. Interestingly, the proof of this is quite different from the proof of Lemma \ref{lemma:main}.

\begin{lemma}\label{lemma:very_dense}
Let $0<\eps<\alpha$, then the following holds if $n$ is sufficiently large. Let $G$ be an $n$-vertex graph of edge density at least $1-n^{-\alpha}$. If $\surp^*(G)\leq n^{1+\eps}$, then $G$ contains an induced subgraph of density at least $(1-O((\log n)^2n^{2\eps-1}))$ on at least $n/2$ vertices.
\end{lemma}

In order to prove this, it is more convenient to work with the complement $\overline{G}$ of $G$. Let $\lambda_1\geq \dots\geq \lambda_n$ be the eigenvalues of $G$, and let $\mu_1\geq \dots\geq \mu_n$ be the eigenvalues of $\overline{G}$. Unfortunately, as $G$ is not necessarily regular, there is no simple formula to express $\mu_i$ in terms of $\lambda_1,\dots,\lambda_n$. However, we can use Weyl's inequality to establish the following interlacing property.

\begin{lemma}\label{lemma:weyl}
Let $G$ be an $n$ vertex graph with eigenvalues $\lambda_1\geq \dots\geq \lambda_n$, and let $\mu_1\geq \dots\geq \mu_n$ be the eigenvalues of the complement of $G$. For $i=1,2\dots,n-1$, $$1+\mu_{i+1}\leq -\lambda_{n+1-i}.$$
\end{lemma}

\begin{proof}
Weyl's inequality states that if $X$ and $Y$ are $n\times n$ symmetric matrices, and $1\leq i,j\leq n$ and $i+j\leq n+1$, then $$\lambda_{i+j-1}(X+Y)\leq \lambda_i(X)+\lambda_j(Y),$$
where $\lambda_1(X)\geq \dots\geq \lambda_n(X)$ denote the eigenvalues of a matrix $X$. Let $A$ be the adjacency matrix of $G$, and let $B$ be the adjacency matrix of $\overline{G}$. Then $B=J-I-A$. Let $X=-A$ and $Y=J-I$, then $\lambda_i(X)=-\lambda_{n+1-i}$, $\lambda_1(Y)=n-1$, $\lambda_i(X)=-1$ for $i=2,\dots,n$, and $\lambda_i(X+Y)=\mu_i$. Hence, applying the above inequality with $j=2$, we get
$$\mu_{i+1}\leq -\lambda_{n+1-i}-1.$$
\end{proof}

The next lemma provides a bound on the surplus of very dense graphs. The result and its proof are similar to the proof of Lemma 5.9 of R\"aty, Sudakov and Tomon \cite{RST} for the complementary quantity called as \emph{positive discrepancy}. Say that a graph $G$ is \emph{$C$-balanced} if $\Delta(G)\leq Cd(G)$. 

\begin{lemma}\label{lemma:main_very_dense}
Let $G$ be an $n$-vertex graph of density $(1-p)$ such that the complement of $G$ is $C$-balanced, and $p<0.001C^{-2}$. Then 
$$\surp^*(G)\geq \Omega\left(\min\left\{\frac{n}{C^3p},C^{-1}p^{1/2}n^{3/2}\right\}\right).$$ 
\end{lemma}

\begin{proof}
Let $A$ be the adjacency matrix of $G$ with eigenvalues $\lambda_1\geq \dots\geq \lambda_n$, and let $B$ be the adjacency matrix of $\overline{G}$ with eigenvalues $\mu_1\geq\dots\geq \mu_n$. Let $\Delta$ be the maximum degree of $\overline{G}$, so $\mu_1\leq \Delta\leq Cpn$. We may assume that $p>0$ and thus $\Delta\geq 1$, otherwise the statement is trivial. For $k=1,2,3$, set $$P_k=\sum_{i\neq 1, \mu_i>0}\mu_i^k\mbox{\ \ \ and\ \ \ } N_k=\sum_{\mu_i<0}|\mu_i|^k.$$ Lemma \ref{lemma:weyl} applied with $i\geq 2$ shows that whenever $\mu_i\geq 0$ we also have $\lambda_{n+1-i}\leq -\mu_i-1<0$. Combined with Lemma~\ref{lemma:surp_star}, this shows that
\begin{align*}
\surp^*(G)&\geq \sum_{\lambda_i<0}|\lambda_i|\geq \sum_{i\neq 1, \mu_i>0}\mu_i = P_1,\\
 \surp^*(G)&\geq \Omega\left(\frac{1}{\sqrt{\Delta}}\sum_{\lambda_i<0}|\lambda_i|^2\right)\geq\Omega\left(\frac{1}{\sqrt{\Delta}}\sum_{i\neq 1, \mu_i>0}\mu_i^2\right) = \Omega\left(\frac{1}{\sqrt{\Delta}}P_2\right),\\
\surp^*(G)&\geq \Omega\left(\frac{1}{\Delta}\sum_{\lambda_i<0}|\lambda_i|^3\right)\geq \Omega\left(\frac{1}{\Delta}\sum_{\mu_i>0}\mu_i^3\right)= \Omega\left(\frac{1}{\Delta}P_3\right).
\end{align*}
We show that these three inequalities together with some simple identities suffice to prove the lemma. 

First, assume that $N_2\leq \frac{1}{8}pn^2$. Note that $\mu_1^2+P_2+N_2=\|B\|_F^2$ is twice the number of edges of $\overline{G}$, so $\mu_1^2+P_2+N_2=2p\binom{n}{2}$, from which
$$P_2\geq pn^2/2-\mu_1^2-N_2\geq pn^2/2-C^2p^2n^2-pn^2/8\geq pn^2/4,$$
where we have used that $pC^2\leq 10^{-3}$ in the last inequality. But then $\surp^*(G)=\Omega(p^{1/2}n^{3/2})$ by the second highlighted inequality, and we are done.

Hence, in the rest of the proof, we may assume that $N_2\geq \frac{1}{8}pn^2$. By the inequality between the quadratic and cubic mean, we have
$$\left(\frac{N_2}{n}\right)^{1/2}\leq \left(\frac{N_3}{n}\right)^{1/3}$$
which gives $N_3\geq N_2^{3/2}n^{-1/2}\geq p^{3/2}n^{5/2}/64.$

Next, consider the quantity $T=N_3-P_3$. Observe that $\mu_1^3-T=\sum_{i=1}^n\mu_i^3$ is six-times the number of triangles of $\overline{G}$. In particular, $\mu_1^3-T$ it is nonnegative, showing that $T\leq \mu_1^3\leq \Delta^3$. Assume that $N_3\geq 2T$, then $P_3\geq N_3/2$. By the third highlighted inequality, we then have $$\surp^*(G)\geq \Omega\left(\frac{P_3}{\Delta}\right)\geq \Omega\left(\frac{N_3}{\Delta}\right)\geq \Omega(C^{-1}p^{1/2}n^{3/2}).$$
Hence, we are done in this case as well.

Finally, assume that $N_3\leq 2T$, then $\Delta^3\geq T\geq N_3/2$. By the Cauchy-Schwartz inequality applied to the sequences $(|\mu_i|^3)_{\mu_i<0}$ and $(|\mu_i|)_{\mu_i<0}$, we have the inequality $N_1N_3\geq N_2^2$, which gives 
$$N_1\geq \frac{N_2^2}{N_3}\geq \frac{p^2n^4}{128\Delta^3}\geq \frac{n}{128C^3p}.$$
But $0=\trace(B)=\mu_1+P_1-N_1$, from which $P_1=N_1-\mu_1\geq \frac{n}{128C^3p}-\Delta\geq \frac{n}{500C^3p}$. Hence, as $\surp^*(G)\geq P_1$, we are done. 
\end{proof}

Next, we present a simple technical lemma which shows that every graph contains a large induced $O(\log n)$-balanced graph.

\begin{lemma}\label{lemma:balanced}
Let $G$ be an $n$-vertex graph of edge density $p$, and let $C\geq 4\log_2 n$. Then $G$ contains a $C$-balanced induced subgraph on at least $(1-2\log_2 n/C)n$ vertices of edge density at most $p$.
\end{lemma}

\begin{proof}
Let $d=d(G)$ be the average degree of $G$. We perform an inductive process, where in each step we delete the vertices of high degree. More precisely, let $G_0=G$, and define the sequence of induced subgraphs $G_0\supset G_1\supset...$ as follows. Having already defined $G_i$, we denote its number of vertices by $n_i$ and its average degree by $d_i$. To define $G_{i+1}$, remove from $G_i$ all vertices of degree at least $Cd_i/2$, if any such vertices exist. The process halts once either all vertices of $G_i$ have degree less than $Cd_i/2$, or the average degree of $G_i$ is at least $d_{i-1}/2$. Let $I$ be the last index $i$ for which $G_i$ is defined. 

For each $i$, it is not hard to verify that the density of $G_i$ is smaller than the density of $G_{i-1}$. Moreover, since the average degree of $G_i$ is $d_i$, there are at most $2n_i/C$ vertices of degree larger than $Cd_i/2$ in $G_i$, so $v(G_i)\geq v(G_{i-1})-2n_i/C\geq v(G_{i-1})-2n/C$. Thus, we have $v(G_i)\geq n(1-2i/C)$ for all $i=1, \dots, I$. Furthermore, if the process did not halt at index $i$, we have $d_i\leq d_{i-1}/2$, and so $d_i\leq d2^{-i}$. This shows that $I\leq \log_2 n$ and $n_I\geq n(1-2\log_2 n/C)$. 

Finally, note that $G_I$ has no vertices of degree more than $C d(G_I)$. Indeed, if the process has halted because $G_I$ contains no vertices of $Cd_I/2$, this is immediate, and if the process has halted because $d_I\geq d_{I-1}/2$, then $G_I$ contains no vertex of degree more than $Cd_{I-1}/2\leq Cd_I$. Hence, $G_I$ is a $C$-balanced induced subgraph of $G$ on at least $(1-2\log_2/C)n$ vertices.
\end{proof}

\begin{proof}[Proof of Lemma \ref{lemma:very_dense}]
Let $1-p$ the edge density of $G$, and let $C=4\log_2 n$. Applying Lemma \ref{lemma:balanced} to the complement of $G$, we find an induced subgraph $G_0\subseteq G$ on at least $n/2$ vertices with edge density $1-p_0\geq 1-p$ such that the complement of $G_0$ is $C$-balanced. As $p_0\leq p\leq 0.001C^{-3}$, we can apply Lemma \ref{lemma:very_dense} to conclude that 
$$\surp^*(G_0)=\Omega\left(\min\left\{\frac{n}{C^3p_I},C^{-1}p_I^{1/2}n^{3/2}\right\}\right).$$
However, as $p_0<n^{-\alpha}$ and $\eps<\alpha$, the inequality $\surp^*(G_0)\leq n^{1+\eps}$ is only possible if 
$$p_0\leq O(C^2n^{2\eps-1})=O(n^{2\eps-1}(\log n)^2).$$
\end{proof}

We can combine Theorem \ref{thm:zhang}, Lemma \ref{lemma:main} and Lemma \ref{lemma:very_dense} into the following ``master lemma'', which shows that somewhat dense graphs contain large cliques.

\begin{lemma}\label{lemma:master}
Let $0<\eps<10^{-3}$. Let $G$ be a graph on $n$ vertices  with at least $n^{2-\eps}$ edges such that $\surp^*(G)\leq n^{1+\eps}$. Then $G$ contains a clique of size $n^{1-20\eps}$.
\end{lemma}

\begin{proof}
 Let $p=10^{-5}$. As $e(G)\geq n^{2-\eps}$ and $\surp(G)\leq \surp^*(G)\leq n^{1+\eps}$, Theorem \ref{thm:zhang} implies that there is an induced subgraph $G_1\subset G$ of edge density at least $1-p$ on $n_1$ vertices, where $n_1>n^{1-4\eps}$. Set $\eps_1=8\eps$ and $\alpha_1=1/20$. Then $\surp^*(G_1)\leq \surp^*(G)\leq n^{1+\eps}\leq n_1^{1+\eps_1}$. Moreover, the conditions of Lemma \ref{lemma:main} are satisfied for $\alpha_1,\eps_1$ and $G_1$, so we can find an induced subgraph $G_2\subset G_1$ of edge density at least $n_1^{-\alpha_1}$ on $n_2$ vertices, where $n_2=n_1^{1-o(1)}$. We have $\surp^*(G_2)\leq \surp^*(G_1)\leq n_2^{1+\eps_1+o(1)}$ and the edge density of $G_2$ is at least $1-n_2^{-\alpha_1}$. As $\eps_1<\alpha_1$, we can apply Lemma \ref{lemma:very_dense} to find an induced subgraph $G_3\subset G_2$ on at least $n_2/2$ vertices of edge density at least $1-O((\log n_2)^2n_2^{2\eps_1-1})$.
But then by Tur\'an's theorem, $G_3$ contains a clique of size at least 
$$\Omega(n_2^{1-2\eps_1}(\log n_2)^{-2})\geq n^{(1-4\eps)(1-16\eps)-o(1)}\geq n^{1-20\eps}.$$
\end{proof}

Finally, we are ready to prove Theorem \ref{thm:maxcut}.

\begin{proof}[Proof of Theorem \ref{thm:maxcut}]
We prove that if $\surp(G)\leq m^{1/2+\eps}$ for some $0<\eps<10^{-5}$, then $G$ contains a clique of size $m^{1/2-30\eps}$. Let $n$ be the number of vertices of $G$. We may assume that $G$ contains no isolated vertices. Then, a result of Erd\H{o}s, Gy\'arf\'as, and Kohayakawa \cite{EGyK} implies that $\surp(G)\geq n/6$. If $m\leq n^{2-3\eps}$, then $\surp(G)\geq n/6\geq \frac{1}{6}m^{1/(2-3\eps)}\geq m^{1/2+\eps}$, contradiction. Hence, $m\geq n^{2-3\eps}$ and $\surp(G)\leq m^{1/2+\eps}\leq n^{1+3\eps}$. But then $G$ contains a clique of size $n^{1-60\eps}>m^{1/2-30\eps}$ by Lemma \ref{lemma:master} applied with $3\eps$ instead of $\eps$.
\end{proof}

\section{Stability of graphs with small surplus}\label{sect:stability}

In this section, we prove Theorem \ref{thm:stability}. To aid the interested reader, we give a brief overview of the proof. Let $G$ be a graph on $n$ vertices with $\surp^*(G)\leq n^{1+\eps}$ for some small $\eps>0$. Then, we use Lemma \ref{lemma:master} to repeatedly pull out large cliques $C_1,\dots,C_I$. We show that we can repeat this procedure until the subgraph induced by $C_1\cup\dots \cup C_I$ contains all but $n^{2-\eps}$ edges of $G$, see Lemma \ref{lemma:many_cliques}. In the proof of this lemma, we employ a technical result showing that graphs with small surplus and large minimum degree cannot contain sparse parts, see Lemma \ref{lemma:inbalanced_partition}. 

Then, we show that the bipartite graph between $C_i$ and $C_j$ for every $1\leq i<j\leq I$ must be either close to complete or close to empty, this can be found in Lemma \ref{lemma:bipartite_complement}. Finally, we finish the proof by observing that we cannot have three different cliques, $C_i,C_j,C_k$, such that $G[C_i\cup C_j]$ and $G[C_j\cup C_k]$ is close to complete, but the bipartite graph between $C_i$ and $C_k$ is close to empty. This then easily implies that $G[C_1\cup\dots\cup C_I]$ is close to the disjoint union of cliques, and we are done.

We start with the following simple lemma, which will be used to argue that  a dense graph with small surplus cannot induce sparse subgraphs.

\begin{lemma}\label{lemma:inbalanced_partition}
Let $G$ be a graph on $n$ vertices. Let $X\cup Y$ be a partition of $V(G)$, and let $b=e(G[X,Y])$ and $c=e(G[Y])$. Then $\surp^*(G)\geq \frac{b^2}{4n^2}-c$.
\end{lemma}

\begin{proof}
If $a=e(G[X])$ satisfies $a\leq b/2$, then $\surp^*(G) \ge \surp(G)$ is at least $$e(G[X,Y])-\frac{e(G)}{2}=b-\frac{a+b+c}{2}= \frac{b-a-c}{2} \ge \frac{b}{4}-\frac{c}{2}\ge \frac{b^2}{4n^2}-c,$$ as desired.

Otherwise, we have $b < 2a$ and we can take $p=b/(4a) \in [0,1/2)$. Let $U$ be a random sample of $X$, where each vertex is sampled independently with probability $1/2+p$, and consider the cut $(U, (X\setminus U)\cup Y)$. Each edge in $G[X]$ has probability $1/2-2p^2$ of being cut, and each edge between $X$ and $Y$ is cut with probability $1/2+p$. Therefore, the expected size of this cut is $a(1/2-2p^2)+b(1/2+p)$, showing that the expected surplus is 
\[a\Big(\frac{1}{2}-2p^2\Big)+b\Big(\frac{1}{2}+p\Big)-\frac{a+b+c}{2}=bp-2ap^2-\frac{c}{2}=\frac{b^2}{8a}-\frac{c}{2} \ge \frac{b^2}{4n^2}-c,\]
where we have used that $a=e(G[X])\leq n^2/2$ in the last step.
\end{proof}

Next, we show that a graph with small surplus contains a collection of large cliques such that almost all edges are contained in the subgraph induced by the union of these cliques.

\begin{lemma}\label{lemma:many_cliques}
For every $\delta>0$ there exists $\eps>0$ such that the following holds. Let $G$ be a graph on $n$ vertices such that $\surp^*(G)\leq n^{1+\eps}$. Then there exists $X\subset V(G)$ such that the number of edges not in $G[X]$ is at most $n^{2-\eps}$, and $G[X]$ can be partitioned into cliques of size $n^{1-\delta}$. 
\end{lemma}

\begin{proof}
We may assume $\delta<10^{-3}$. Define $\delta_0=\delta/2$ and $\eps_0=\delta_0/20$. By Lemma \ref{lemma:master}, every $n_0$-vertex graph with at least $n_0^{1-\eps_0}$ edges and surplus at most $n_0^{1+\eps_0}$ contains a clique of size $n_0^{1-\delta_0}$. 

We now prove the statement with $\eps=\eps_0/10$ suffices. Delete all vertices of $G$ of degree less than $d=n^{1-2\eps}$, and let $G_0$ be the resulting graph. Note that we removed at most $dn$ edges. Repeat the following procedure. If $G_i$ is defined and $G_i$ contains a clique of size $n^{1-\delta}$, then let $C_{i+1}$ be such a clique, and set $G_{i+1}=G_i\setminus C_{i+1}$. Otherwise, stop, and let $I$ be the last index $i$ for which $G_i$ is defined. We show that $X=C_1\cup\dots\cup C_I$ suffices. It is clear that $X$ can be partitioned into cliques of size $n^{1-\delta}$, so it remains to show that the number of edge not in $X$ is at most $n^{2-\eps}$. 

Let $Y=V(G_I)$, $b=e(G_0[X,Y])$ and $c=e(G_0[Y])$.
\begin{claim}
    If $|Y|\geq n^{2/3}$, then $c\geq d^2|Y|^2/(20n^2)$.
\end{claim}
\begin{proof}
    As $G$ has minimum degree $d$, we have $b+2c\geq d|Y|$. 
    Assume that $c<d^2|Y|^2/(20n^2)<d|Y|/4$, then $b\geq d|Y|/2$. 
    By Lemma \ref{lemma:inbalanced_partition}, we have
    $$
        n^{1+\eps}\geq \surp^*(G)\geq \surp^*(G_0)\geq \frac{b^2}{4n^2}-c,
    $$
    from which 
    $$c\geq \frac{b^2}{4n^2}-n^{1+\eps}\geq \frac{d^2|Y|^2}{16n^2}-n^{1+\eps}\geq \frac{d^2|Y|^2}{20n^2}.$$
    In the last inequality we used that $d>n^{3/4}$ and $|Y|\geq n^{2/3}$.
\end{proof}
Hence, if $|Y|\geq n^{1-2\eps}$, then $G_I$ has at least $d^2|Y|^2/(20n^2)>|Y|^2n^{-5\eps}\geq |Y|^{2-\eps_0}$ edges and $\surp^*(G_I)\leq \surp^*(G)\leq n^{1+\eps}\leq |Y|^{1+4\eps}\leq |Y|^{1+\eps_0}$. But then $Y$ contains a clique of size $|Y|^{1-\delta_0}>n^{1-2\eps-\delta_0}>n^{1-\delta}$, contradicting that $G_I$ contains no clique of size $n^{1-\delta}$. Therefore, we must have $|Y|\leq n^{1-2\eps}$.

From this, the number of edges of $G$ not in $G[X]$ is at most 
$$dn+e(G[X,Y])+e(G[Y])\leq dn+|Y|n\leq n^{1-\eps}.$$
This finishes the proof.
\end{proof}

Our next goal is to show that if $H$ is a bipartite graph, then the complement of $H$ has large surplus, unless $H$ is close to the empty or complete bipartite graph. We prepare the proof of this with two lemmas.

A \emph{Boolean matrix} is a matrix with only zero and one entries. We show that if a Boolean  matrix is approximated by a rank one matrix, then it is also approximated by a rank one Boolean matrix, or equivalently, a combinatorial rectangle.

\begin{lemma}\label{lemma:rank1_approximation}
Let $A$ be an $n\times n$ Boolean matrix, and let $\delta\geq 0$. If there exist $u,v\in \mathbb{R}^n$ such that $\|A-uv^T\|_F^2\leq \delta n^2$, then there exist $x,y\in \{0,1\}^n$  such that $\|A-xy^T\|_F^2\leq O(\delta^{1/3}n^2)$.
\end{lemma}

\begin{proof}
 Without loss of generality, we may assume that $\delta\leq 1$. Furthermore, we may assume that $u$ and $v$ has nonnegative entries, as replacing every entry with the absolute value does not increase $\|A-uv^T\|_F^2$. Observe that $\|u\|_2^2\|v\|_2^2=\|uv^T\|_F^2$, which shows that $$\|u\|_2\|v\|_2\leq \|A\|_F+\sqrt{\delta}n\leq 2n.$$
We may rescale $u$ and $v$ such that $\|u\|_2=\|v\|_2\leq \sqrt{2n}$. Let $\alpha=\delta^{1/6}$, and define $x,y\in \{0,1\}^n$ such that 
 $$x(i)=\begin{cases} 
 1 &\mbox{ if } u(i)\geq \alpha \\
 0 &\mbox{ otherwise,}\end{cases}$$
 and similarly
  $$y(i)=\begin{cases} 
 1 &\mbox{ if } v(i)\geq \alpha \\
 0 &\mbox{ otherwise.}\end{cases}$$
 We show that $xy^T$ is a good approximation of $A$. Note that $\|A-xy^T\|_F^2$ is the number of pairs $(i,j)$ such that $A_{i,j}\neq x_iy_j$. We count these pairs in three cases.
 \begin{description}
     \item[Case 1.] $A_{i,j}=1$ and $x_i=0$.

     In this case, we have $u_i< \alpha$. If $v_j\leq 1/(2\alpha)$, then $(A_{i,j}-u_iv_j)^2>1/4$, so there are at most $4\delta n^2$ such pairs $(i,j)$. On the other hand, the number of $j$ such that $v_j\geq 1/(2\alpha)$ is at most $4\alpha^2 n$, as $\|v\|_2^2=\sum_{j=1}^nv_j^2\leq 2n$. Therefore, the number of $(i,j)$ such that $A_{i,j}=1$ and $x_i=0$ is at most $$4\delta n^2+4\alpha^2n^2=4\delta n^2+4\delta^{1/3}n^2=O(\delta^{1/3}n^2).$$

    \item[Case 2.] $A_{i,j}=1$ and $y_j=0$.
    
    This is symmetric to the previous case, so the number of such pairs is also at most $O(\delta^{1/3}n^2)$.
     
     \item[Case 3.] $A_{i,j}=0$ and $x_i=y_j=1$.

     In this case, $u_i\geq \alpha$ and $v_j\geq \alpha$, so $(A_{i,j}-u_iv_j)^2\geq \alpha^4$. Thus, the total number of pairs $(i,j)$ in this case is at most $\delta n^2/\alpha^4=\delta^{1/3}n^2$.\qedhere
 \end{description}
\end{proof}

Next, we prove a simple technical lemma which shows that the union of two cliques has large surplus as long as the two cliques are not too disjoint, and do not overlap too much. 

\begin{lemma}\label{lemma:union_of_cliques}
Let $G$ be a graph such that $V(G)=C_1\cup C_2$ and $E(G)=\binom{C_1}{2}\cup \binom{C_2}{2}$. Let $|C_1\setminus C_2|=a$, $|C_2\setminus C_1|=b$ and $|C_1\cap C_2|=c$. Then $$\surp(G)\geq \frac{1}{4}\min\{a^2,b^2,c^2\}.$$ 
\end{lemma}

\begin{proof}
Let $A=C_1\setminus C_2$, $B=C_2\setminus C_1$, and $C=C_1\cap C_2$. We may assume that $a=b$. Otherwise, if, say $a\leq b$, we remove vertices of $B\setminus C$ until its size is exactly $a$. Then it is enough to show that the resulting graph has surplus at least $\frac{1}{4}\min\{a^2,c^2\}$.

The number of edges of $G$ is 
$$2\binom{a+c}{2}-\binom{c}{2}< a^2+2ac+\frac{c^2}{2}.$$ 
If $c\leq a$, then define the cut $(U,V)$ such that $U$ is some $(a+c)/2$ element subset of $A$ together with some $(a+c)/2$ element subset of $B$. The number of edges in this cut is $(a+c)^2/2$. Hence, the surplus of $G$ is at least $\frac{1}{2}(a+c)^2-\frac{1}{2}e(G)=c^2/4$.

If $c\geq a$, then define the cut $(U,V)$ such that $U$ is some $(a+c)/2$ element subset of $C$. Then the number of edges in this cut is $\frac{a+c}{2}\cdot \frac{c-a}{2}+2\frac{a+c}{2}\cdot a=\frac{c^2}{4}+\frac{3}{4}a^2+ac$. Therefore, the surplus is at least $a^2/4$.
\end{proof}

Now we are ready to prove our lemma about the surplus of the complement of bipartite graphs.

\begin{lemma}\label{lemma:bipartite_complement}
Let $\eps,\delta>0$ be parameters such that $\eps+6\delta<1/2$. Let $H$ be a bipartite graph with vertex classes of size $n$, and let $G=\overline{H}$. If  $\surp^*(G)\leq n^{1+\eps}$, then either $e(H)\leq n^{2-\delta}$ or $e(H)\geq n^2-n^{2-\delta}$.
\end{lemma}

\begin{proof}
Let $A$ be the adjacency matrix of $G$ with eigenvalues $\lambda_1\geq \dots\geq \lambda_{2n}$. Furthermore, let $M$ be the adjacency matrix of $H$, and let $\mu_1\geq \dots\geq \mu_{2n}$ be the eigenvalues of $M$. As $H$ is bipartite, we have $\mu_i=-\mu_{2n+1-i}$ for $i\in [2n]$. 
By Lemma \ref{lemma:surp_star} \textit{(ii)} and Lemma \ref{lemma:weyl}, we have
$$\surp^*(G)\geq \Omega\left(\frac{1}{\sqrt{n}}\sum_{\lambda_i<0}\lambda_i^2\right)\geq \Omega\left(\frac{1}{\sqrt{n}}\sum_{i\neq 1,\mu_i>0}\mu_i^2\right)=\Omega\left(\frac{1}{\sqrt{n}}\sum_{i\neq 1,2n}\mu_i^2\right).$$
Hence, if $\surp^*(G)\leq n^{1+\eps}$, then we have $\sum_{i\neq 1, 2n}\mu_i^2=O(n^{3/2+\eps})$.

On the other hand, we can express $\sum_{i\neq 1, 2n}\mu_i^2$ as follows. The matrix $M$ has the form $M=\begin{pmatrix}0 & B \\ B^T & 0\end{pmatrix}$ with an appropriate $n\times n$ matrix $B$. Let $v_1$ be the principal eigenvector  of $M$, then we can write $v_1=(u,v)$, where $u,v\in\mathbb{R}^n$ correspond to the two vertex classes of $H$. Then the eigenvector corresponding to the smallest eigenvalue $\lambda_{2n}=-\lambda_1$ is $v_{2n}=(u,-v)$, and we have
\begin{align*}
\sum_{i\neq 1,2n}\mu_i^2&=\|M-\lambda_1v_1v_1^T-\lambda_{2n}v_{2n}v_{2n}^T\|_F^2\\
&=\left\|\begin{pmatrix} 0 & B \\ B^T & 0\end{pmatrix}-\lambda_1\begin{pmatrix} uu^T & uv^T \\ vu^T & vv^T\end{pmatrix}+\lambda_1\begin{pmatrix} uu^T & -uv^T \\ -vu^T & vv^T\end{pmatrix}\right\|_F^2=2\|B-2\lambda_1uv^T\|_F^2.
\end{align*}
Therefore, $\|B-2\lambda_1vu^T\|_F=O(n^{3/2+\eps})$. But then by Lemma \ref{lemma:rank1_approximation}, there exist $x,y\in \{0,1\}^n$ such that $\|B-xy^T\|_F^2=O(n^{5/6+\eps/3})$. The matrix $xy^T$ corresponds to a complete bipartite graph between the vertex classes of $H$, let $\widetilde{H}$ denote this complete bipartite graph, and let $X_0$ and $Y_0$ denote its vertex classes. Note that $e(\widetilde{H})=\|xy^T\|_F^2=|X_0||Y_0|$, and $\|B-xy^T\|_F^2$ is the number of edges $\widetilde{H}$ differs from $H$.  Therefore, if $e(\widetilde{H})\leq n^{2-\delta}/2$, then $e(H)\leq e(\widetilde{H})+\|B-xy^T\|_F\leq n^{2-\delta}$, so we are done. We can proceed similarly if $e(\widetilde{H})\geq n^2-n^{2-\delta}/2$. Hence, we may assume that $n^{2-\delta}/2\leq e(\widetilde{H})\leq n^2-n^{2-\delta}/2$. We show that this is impossible, by deriving that the surplus of $G$ is too large in this case.

Let $\widetilde{G}$ be the complement of $\widetilde{H}$. Then $\widetilde{G}$ and $G$ differ by at most $O(n^{5/6+\eps/3})$ edges. On the other hand, $\widetilde{G}$ is the union of two cliques, having vertex sets $C_1$ and $C_2$, where $X_0=C_1\setminus C_2$, $Y_0=C_2\setminus C_1$, and $C_1\cap C_2=V(G)\setminus (X_0\cup Y_0)$. As $n^{2-\delta}/2\leq e(\widetilde{H})=|X_0||Y_0|$, we have $|X_0|,|Y_0|\geq n^{1-\delta}/2$. Also, as $e(\widetilde{H})\leq n-n^{2-\delta}/2$, we have $|C_1\cap C_2|=|V(G)\setminus (X_0\cup Y_0)|\geq n^{2-\delta}/2$. Hence, by applying Lemma \ref{lemma:union_of_cliques}, we get that $\surp^*(\widetilde{G})\geq \Omega(n^{2-2\delta})$. But as $G$ and $\widetilde{G}$ differ by less than $O(n^{5/6+\eps/3})$ edges, and $5/6+\eps/3<2-2\delta$, this gives 
$$\surp^*(G)\geq \surp^*(\widetilde{G})-O(n^{5/6+\eps/3})\geq \Omega(n^{2-2\delta})>n^{1+\eps}$$ as well, contradiction.
\end{proof}

After these preparations, we are ready to prove the main theorem of this section.

\begin{proof}[Proof of Theorem \ref{thm:stability}]
Let $\delta_0=10^{-3}$, and let $\eps=\min\{\eps_0/2,10^{-4}\}$, where $\eps_0$ is the constant guaranteed by Lemma \ref{lemma:many_cliques} with respect to $\delta=\delta_0$. We show that $\eps$ suffices. Let $G$ be a graph on $n$ vertices such that $\surp^{*}(G)\leq n^{1+\eps}$. Then Lemma \ref{lemma:many_cliques} guarantees a set $X\subset V(G)$ such that $X$ can be partitioned into the union of cliques of size $n_0=n^{1-\delta_0}$, and $G$ has at most $n^{2-\eps_0}$ edges not in $G[X]$. Let $C_1,\dots,C_I$ be the cliques of size $n_0$ partitioning $X$, then $I=|X|/n_0$. 

Let $\eps_1=\delta_1=0.05$, then $\eps_1+6\delta_1<1/2$. For every $0\leq i<j\leq I$, let $G_{i,j}=G[C_i\cup C_j]$. Then $\surp^*(G_{i,j})\leq \surp^*(G)\leq n^{1+\eps}\leq n_0^{1+\eps_1}$, so Lemma \ref{lemma:bipartite_complement} implies that the bipartite graph between $C_i$ and $C_j$ has either at most $n_0^{2-\delta_1}$, or at least $n_0^2-n_0^{2-\delta_1}$ edges. Define the auxiliary graph $\Gamma$ on vertex set $\{1,\dots,I\}$, where we connect $i$ and $j$ if the bipartite graph between $C_i$ and $C_j$ has at least $n_0^2-n_0^{2-\delta_1}$ edges.

\begin{claim}
$\Gamma$ contains no cherry, that is, three vertices $i,j,k$ such that $ij,jk\in E(\Gamma)$, but $ik\not\in E(\Gamma)$.
\end{claim}

\begin{proof}
Assume to the contrary that $i,j,k\in \{1,\dots,I\}$ forms a cherry. Let $G'$ be the subgraph induced on the vertex set $C_i\cup C_j\cup C_k$. Consider the cut $(C_j,C_i\cup C_k)$ in $G'$. The size of this cut is at least $2n_0^2-2n_0^{2-\delta_1}\geq 1.9n_0^2$. But $e(G')\leq \frac{7}{2}n_0^{2}+n_0^{2-\delta_1}\leq 3.6n_0^2$. Therefore, $n^{1+\eps}\geq \surp^*(G)\geq \surp^*(G')\geq 0.1n_0^2$, contradiction. 
\end{proof}

It is easy to show that graphs containing no cherry are the disjoint union of cliques. Therefore, we can partition $V(\Gamma)$ into sets $I_1,\dots,I_s$ such that $\Gamma[I_i]$ is a clique and there are no edges between $I_i$ and $I_j$ in $\Gamma$. But this gives a partition of $X$ into sets $Y_1,\dots,Y_s$ by setting $Y_a=\bigcup_{i\in I_a}C_i$. Define $\widetilde{G}$ to be the graph on vertex set $V(G)$, where $Y_1,\dots,Y_s$ are cliques, and all edges of $\widetilde{G}$ are contained in one of these cliques.

We prove that $\widetilde{G}$ is $n^{-\eps}$-close to $G$. For $1\leq i<j\leq I$, $G[C_i,C_j]$ and $\widetilde{G}[C_i,C_j]$ differ by at most $n_0^{2-\delta_1}$ edges. Therefore, $G[X]$ and $\widetilde{G}[X]$ differ by at most $$\binom{|X|/n_0}{2}n_0^{2-\delta_1}\leq n^2 n_0^{-\delta_1}\leq n^{2-\delta_1/(1-\delta_0)}\leq n^{2-2\delta_1}$$
edges. Furthermore, there are at most $n^{2-\eps_0}$ edges of $G$ not in $G[X]$, so $G$ and $\widetilde{G}$ differ by at most $n^{2-\eps_0}+n^{2-2\delta_1}\leq n^{2-\eps}$ edges. This finishes the proof.
\end{proof}

\section*{Acknowledgments}
We would like to thank Shengtong Zhang for sharing early versions of his manuscript with us, and Benny Sudakov for many valuable discussions. Also, we would like to thank Victor Falgas-Ravry, whose question ultimately lead to this project.

\end{document}